\documentclass[pdflatex,sn-mathphys-num]{sn-jnl}

\usepackage{amsmath,amssymb,amsthm,mathtools}
\usepackage{graphicx}%
\usepackage{multirow}%
\usepackage{amsmath,amssymb,amsfonts}%
\usepackage{amsthm}%
\usepackage{mathrsfs}%
\usepackage[title]{appendix}%
\usepackage{xcolor}%
\usepackage{textcomp}%
\usepackage{manyfoot}%
\usepackage{booktabs}%
\usepackage{algorithm}%
\usepackage{algorithmicx}%
\usepackage{algpseudocode}%
\usepackage{listings}%
\usepackage{xcolor}
\usepackage{graphicx}
\usepackage{booktabs}
\usepackage{hyperref}
\hypersetup{
  colorlinks=true,
  linkcolor=blue!55!black,
  citecolor=blue!55!black,
  urlcolor=blue!55!black
}

\theoremstyle{plain}
\newtheorem{theorem}{Theorem}
\newtheorem{proposition}[theorem]{Proposition}

\newtheorem{corollary}[theorem]{Corollary}
\theoremstyle{definition}

\newtheorem{assumption}[theorem]{Assumption}
\newtheorem{remark}[theorem]{Remark}
\newtheorem{example}[theorem]{Example}

\newcommand{\R}{\mathbb{R}}
\newcommand{\E}{\mathbb{E}}
\newcommand{\PR}{\mathbb{P}}
\newcommand{\Filt}{\mathbb{F}}
\newcommand{\Fx}{\mathcal{F}}
\newcommand{\Lgen}{\mathscr{L}}
\newcommand{\ind}{\mathbf{1}}
\newcommand{\dd}{\,\mathrm{d}}
\newcommand{\snr}{\vartheta}
\newcommand{\post}{\Pi}
\newcommand{\odds}{\Phi}
\newcommand{\bbar}{\bar{B}}

\begin{document}

\title{Delay-Penalty Comparison for Sequential Testing and Quickest Detection in State-Dependent Diffusion Models}

\author*{\fnm{Ye} \sur{Liang}}\email{ye-liang@uiowa.edu}

\affil{\orgdiv{College of Engineering},
  \orgname{The University of Iowa},
  \orgaddress{\city{Iowa City}, \postcode{52242}, \state{IA},\country{USA}}}
  
\abstract{
We study sequential testing and Bayesian quickest detection for diffusion
observations whose drift changes between two alternatives while the
signal-to-noise ratio may depend on the current observation. In this setting the posterior probability is generally not a closed one-dimensional Markov statistic: the natural sufficient state is the augmented process consisting of the posterior (or likelihood ratio) and the observed diffusion. We formulate both testing and quickest detection within this common filtering framework and identify the corresponding degenerate free-boundary problems. The main contribution is a delay-penalty comparison principle. For a common terminal false-alarm or terminal decision cost, a pointwise larger running delay penalty increases the value of continuation, shrinks the continuation region, and yields earlier stopping. When the stopping set has a one-sided posterior representation, this gives an order relation for the optimal alarm boundaries. The result applies to linear delay costs and to nonlinear marginal delay penalties after the appropriate Markovian augmentation, and is illustrated by a constant signal-to-noise Shiryaev example
in which the alarm threshold is computed numerically and shown to be monotone in the delay cost. The framework clarifies how state-dependent information and nonlinear delay costs jointly affect the geometry of sequential testing and quickest-detection rules.}



\maketitle

\section{Introduction}\label{sec:intro}

Sequential methods determine endogenously when accumulated data justify a terminal action, balancing the cost of further observation against the quality of the eventual decision \citep{wald2004sequential,wald1948optimum,rincon2025sequential,wang2025analysis,chow1971great,griffith2021statistics,silva2020optimal,silva2015continuous,wang2025analysis1,fischer2026improving}.
Two problems organize much of the field. In \emph{sequential testing} an observer
watches a process whose law is governed by one of two simple hypotheses and must
choose, as a function of the data, both a stopping time and a terminal decision so
as to trade sampling cost against the probabilities of a wrong decision \cite{shiryaev2025optimal,pabbaraju2026simple,liu2025bidirectional}. In
\emph{quickest detection}, or the \emph{disorder problem}, the law of the
observation changes at an unobservable time, and the observer raises an alarm to
trade the frequency of false alarms against the delay incurred in detecting a true
change \citep{shiryaev1963optimum,naha2023quickest,snow2024quickest,sha2025quickest,liang2025global}. Minimax counterparts of the detection
problem replace the prior on the change time by a worst-case criterion and lead to
the CUSUM and Shiryaev--Roberts procedures
\citep{page1954continuous,banerjee2024minimax,xie2022minimax,wang2026algebraic,fromont2023minimax,yang2024sequential,huselitz2026online,lorden1971procedures,tosun2023robust,yu2024network,moustakides1986optimal,pollak1985optimal,wang2026damage,polunchenko2018comparative,pollak2009optimality}. The
present paper is Bayesian and optimal-stopping based.

Both problems arise across applied domains. In quantitative finance a regime
shift in a return or volatility process must be detected promptly while
controlling false alarms; in statistical process control a manufacturing stream
must be monitored for a shift in mean; in structural health monitoring,
surveillance, and intrusion detection a sensor stream must be screened for the
onset of an anomaly; and in epidemiological monitoring an incidence series must be
watched for the start of an outbreak. In each case the natural observation model
is a continuously sampled diffusion, the cost of a missed or delayed detection is
problem specific and frequently nonlinear, and the practitioner needs to
understand how the rule---and in particular the alarm threshold---moves when the
penalty structure changes. That comparative question, rather than the explicit
solution of any single model, is the focus of this paper.

By filtering, both problems reduce to fully observed optimal stopping for a
Markovian sufficient statistic, after which the value function solves a
variational inequality and the optimal rule is a first-exit time from a
continuation region \citep{liptser1977statistics,peskir2006optimal,yu2026pattern,peskir2000sequential,epstein2022optimal,ekstrom2022multi,ankirchner2020bayesian,wang2026breakdown,shiryaev2025optimal}.
For Brownian observations with constant drift alternatives the sufficient
statistic is a one-dimensional posterior diffusion and the rules are explicit
thresholds. The situation changes when the diffusion coefficients are state
dependent: the \emph{signal-to-noise ratio} then depends on the current
observation, the posterior equation no longer closes in the posterior coordinate
alone, and the sufficient statistic becomes multidimensional. This is the regime
we study, and it is the regime in which the comparative-statics question is least
understood.

\paragraph{Contribution.}
This paper makes two related contributions. First, we give a unified formulation
of sequential testing and Bayesian quickest detection for state-dependent
diffusion observations. The formulation makes explicit that, unless the
signal-to-noise ratio is constant or otherwise reducible, the posterior
probability is not a closed Markov state; the closed sufficient statistic is the
augmented process $(\post_t,X_t)$, or equivalently $(\odds_t,X_t)$, with an
explicit degenerate generator. Second, and more importantly, we prove a
delay-penalty comparison theorem for the resulting optimal stopping problems.
For a fixed terminal cost, increasing the running delay penalty increases the
value function, shrinks the continuation region, and induces earlier stopping.
Whenever the stopping region is known to be one-sided in the posterior
coordinate, this gives a monotone ordering of the alarm boundaries. This
comparative-static result applies directly to linear delay costs and extends to
nonlinear marginal delay penalties after adding the appropriate penalty-memory
state. A worked Shiryaev example computes the alarm threshold numerically by a
finite-difference solution of the variational inequality and exhibits the
predicted monotonicity.

\paragraph{Organization.}
Section~\ref{sec:filter} develops the filtering reductions for both problems and
proves the closed Markov-state result. Section~\ref{sec:os} states the generic
optimal stopping problem and the variational inequality in complementarity form.
Section~\ref{sec:comparison} proves the delay-penalty comparison theorem and its
linear and nonlinear specializations and gives the sampling-cost analogue for
testing. Section~\ref{sec:fb} gives the free-boundary interpretation and the
martingale verification. Section~\ref{sec:example} is the worked Shiryaev example
with the numerical method described in full. Section~\ref{sec:cusum} relates the
framework to CUSUM and Shiryaev--Roberts procedures, and Section~\ref{sec:concl}
concludes.

\section{Literature Review}

\paragraph{Relation to existing literature.}

The literature relevant to the present work lies at the intersection of
sequential analysis, quickest detection, filtering, and optimal stopping.

On the sequential-analysis side, the foundations were laid by Wald and
coauthors through the theory of sequential testing and optimal stopping of
experiments
\citep{wald2004sequential,wald1948optimum,yu2026rigorous,chow1971great,griffith2021statistics}.
In parallel, Shiryaev initiated the Bayesian disorder problem, in which an
unobservable change point must be detected as rapidly as possible while
controlling false alarms
\citep{shiryaev1963optimum}. 
These developments evolved into the modern theory
of quickest detection, encompassing both Bayesian and minimax
formulations, with comprehensive treatments given by
Shiryaev's monographs and retrospective survey and by the
unified treatment of Poor and Hadjiliadis and others \citep{shiryaev2025optimal,poor2009quickest,shiryaev2010quickest,yu2026beyond,shiryaev2009stochastic,pollak1985diffusion,pollak1985optimal,moustakides1986optimal,shiryaev2019stochastic,cai2026optimal}. The Bayesian formulation is naturally
expressed as a Markov optimal stopping problem after filtering, while minimax
formulations lead to procedures such as CUSUM and Shiryaev--Roberts
\citep{page1954continuous,wang2025multi,lorden1971procedures,moustakides1986optimal,
pollak1985optimal,pollak2009optimality,polunchenko2018comparative,gao2022rolling}. The field
now provides a unified framework for applications ranging from engineering and
finance to surveillance and epidemiology.

A second strand of literature concerns diffusion observations and
state-dependent signal structures. For Brownian models with constant
signal-to-noise ratio, filtering reduces the problem to a one-dimensional
posterior diffusion and optimal rules are characterized by scalar thresholds.
The situation becomes substantially more difficult when the signal-to-noise
ratio depends on the current state of the observed diffusion. In this setting,
the posterior probability alone generally fails to form a closed Markov state.
A line of research initiated by Gapeev and Shiryaev
\citep{gapeev2011sequential,gapeev2013bayesian} developed sequential testing
and quickest detection formulations for diffusion processes with
state-dependent coefficients. Subsequent analyses by Johnson and Peskir
\citep{johnson2017quickest,yu2026from,johnson2018sequential} revealed the rich boundary
structure that may arise even in special cases such as Bessel-process
observations. Most recently, Ernst and Peskir
\citep{ernst2024gapeev} resolved the Gapeev--Shiryaev conjecture by proving
that monotonicity of the signal-to-noise ratio implies monotonicity of the
associated optimal stopping boundaries.

A related body of work studies multidimensional and multi-source detection
problems. When observations are collected from multiple sensors or coupled
systems, the filtering state becomes vector valued and the stopping region is
typically a hypersurface rather than a scalar threshold. Such models arise in
distributed surveillance, sensor networks, and structural monitoring, and
provide natural examples where finite-dimensional sufficient statistics remain
available but the geometry of the stopping rule becomes substantially more
complex \cite{ludkovski2012bayesian,zhang2014quickest,kurt2018multisensor,konev2017quickest,didi2024active,dayanik2016sequential}.

Methodologically, the present paper belongs to the optimal-stopping and
free-boundary literature \cite{peskir2006optimal}. After filtering, both sequential testing and
Bayesian quickest detection reduce to variational inequalities for Markov
sufficient statistics, with optimal rules represented as first-entry times
into stopping regions
\citep{peskir2000sequential,ekstrom2022multi,
epstein2022optimal,ankirchner2020bayesian}. Much of the existing literature
focuses on deriving explicit solutions, characterizing free boundaries, or
establishing structural properties of stopping regions for a fixed penalty
specification. In diffusion-based optimal stopping problem, the principal questions
have traditionally been existence, regularity, smooth fit, and
stochastic control of optimal boundaries \cite{peskir2019continuity,arkin2009variational,oshima2006optimal,oksendal2019applied}.

The present paper addresses a different question. We do not seek a new
explicit solution of a particular diffusion stopping problem. Instead, we
study how the optimal rule changes when the delay-penalty structure changes.
Our object is therefore the comparative-statics map
\[
\text{penalty profile}
\longmapsto
\text{value function}
\longmapsto
\text{continuation region}
\longmapsto
\text{stopping boundary}.
\]
Once a filtering reduction and optimal-stopping formulation are available,
we show that larger marginal delay penalties increase the value function,
shrink continuation regions, induce earlier stopping, and, whenever a
one-sided boundary representation is known, produce a monotone ordering of
alarm boundaries. Thus the contribution of the paper is not a new boundary
formula but a structural comparison principle that applies across a broad
class of diffusion-based sequential testing and quickest-detection models.

\section{Filtering Reductions for Diffusion Sequential Problems}\label{sec:filter}

Throughout, $(\Omega,\Fx,\Filt,\PR)$ carries a standard Brownian motion
$B=(B_t)_{t\ge0}$, the observation is a one-dimensional diffusion
$X=(X_t)_{t\ge0}$ on a state space $I\subseteq\R$, and $\Filt^X=(\Fx^X_t)_{t\ge0}$
is the augmented right-continuous observation filtration. We write
$\PR_i:=\PR(\,\cdot\mid\theta=i)$ for the conditional laws under the two
hypotheses and $\E_i,\E_\pi$ for the corresponding expectations.

\subsection{Sequential testing}\label{sec:test}
Under hypothesis $H_i$, $i\in\{0,1\}$, the observation solves
\begin{equation}\label{eq:Xdyn}
  \dd X_t=\mu_i(X_t)\dd t+\sigma(X_t)\dd B_t,\qquad X_0=x_0,
\end{equation}
with $\mu_0\not\equiv\mu_1$ and $\sigma>0$. The hidden hypothesis
$\theta\in\{0,1\}$ has prior $\pi:=\PR(\theta=1)\in(0,1)$. A testing rule is a
pair $(\tau,d)$ consisting of an $\Filt^X$-stopping time $\tau$ and an
$\Fx^X_\tau$-measurable terminal decision $d\in\{0,1\}$. With unit sampling cost
per unit time and error costs $a,b>0$ for the two types of error, the Bayes risk
is
\begin{equation}\label{eq:risk-test}
  R_\pi(\tau,d)=\E_\pi[\tau]
   +a\,\PR_\pi(d=1,\theta=0)+b\,\PR_\pi(d=0,\theta=1).
\end{equation}
Let $\post_t:=\PR_\pi(\theta=1\mid\Fx^X_t)$ be the posterior probability of $H_1$.
For a fixed stopping time, the terminal-error part of \eqref{eq:risk-test} is
minimized by deciding $H_1$ when its posterior cost is smaller, that is,
$d^\ast=\ind_{\{\post_\tau\ge p^\dagger\}}$ with $p^\dagger=a/(a+b)$, and the
resulting conditional terminal cost is
\begin{equation}\label{eq:Mdef}
  \E_\pi\!\big[a\,\ind_{\{d^\ast=1,\theta=0\}}+b\,\ind_{\{d^\ast=0,\theta=1\}}\mid\Fx^X_\tau\big]
   =\min\{b\post_\tau,\,a(1-\post_\tau)\}=:M(\post_\tau),
\end{equation}
where $M$ is concave and piecewise linear with apex at $p^\dagger$. Substituting
$d^\ast$ collapses \eqref{eq:risk-test} to the optimal stopping problem
\begin{equation}\label{eq:value-test}
  V(\pi,x_0)=\inf_\tau\,\E_{\pi,x_0}\big[\tau+M(\post_\tau)\big].
\end{equation}
Under $\PR_1\ll\PR_0$ on $\Fx^X_t$, Girsanov's theorem gives the likelihood ratio
\begin{equation}\label{eq:LR}
  L_t=\left.\frac{\dd\PR_1}{\dd\PR_0}\right|_{\Fx^X_t}
   =\exp\!\left\{\int_0^t\frac{\mu_1-\mu_0}{\sigma^2}(X_s)\dd X_s
   -\tfrac12\int_0^t\frac{\mu_1^2-\mu_0^2}{\sigma^2}(X_s)\dd s\right\}.
\end{equation}
Introducing the signal-to-noise ratio
\begin{equation}\label{eq:snr}
  \snr(x):=\frac{\mu_1(x)-\mu_0(x)}{\sigma(x)},
\end{equation}
one has $\dd L_t/L_t=\snr(X_t)\dd B_t^0$ under $\PR_0$, where $B^0$ is the
$\PR_0$-driving Brownian motion, so $L$ is a $\PR_0$-martingale. The posterior
odds are $\odds_t:=\post_t/(1-\post_t)=\tfrac{\pi}{1-\pi}L_t$. 
The innovation process
\begin{equation}\label{eq:innov}
  \bbar_t=\int_0^t\frac{1}{\sigma(X_s)}\Big(\dd X_s
    -\big[\mu_0(X_s)+(\mu_1-\mu_0)(X_s)\post_s\big]\dd s\Big)
\end{equation}
is a standard $\Filt^X$-Brownian motion \citep{liptser1977statistics}, and the
Kushner--Stratonovich equation for the two-valued hidden variable gives the
posterior diffusion together with the observation in innovation form,
\begin{equation}\label{eq:post-test-sde}
  \dd\post_t=\snr(X_t)\,\post_t(1-\post_t)\dd\bbar_t,\qquad
  \dd X_t=\big[\mu_0(X_t)+(\mu_1-\mu_0)(X_t)\post_t\big]\dd t+\sigma(X_t)\dd\bbar_t.
\end{equation}
The odds process is the smooth image $\odds=\post/(1-\post)$ of $\post$ under the
bijection $p\mapsto p/(1-p)$ of $(0,1)$ onto $(0,\infty)$; we therefore use
$(\post,X)$ and $(\odds,X)$ interchangeably as state descriptors and do not record
a separate stochastic differential for $\odds$, which carries an It\^o correction
relative to \eqref{eq:post-test-sde}.

\subsection{Bayesian quickest detection}\label{sec:detect}
Now the drift switches at an unobservable change time $\theta\ge0$:
\begin{equation}\label{eq:disorder-dyn}
  \dd X_t=\mu_0(X_t)\dd t+\sigma(X_t)\dd B_t\ \ (t<\theta),\qquad
  \dd X_t=\mu_1(X_t)\dd t+\sigma(X_t)\dd B_t\ \ (t\ge\theta),
\end{equation}
with the standard prior placing an atom at the origin and an exponential tail,
\begin{equation}\label{eq:prior}
  \PR(\theta=0)=\pi,\qquad \PR(\theta>t\mid\theta>0)=e^{-\lambda t},\quad\lambda>0.
\end{equation}
For an alarm time $\tau$ the linear-delay Bayes risk weighs the probability of a
false alarm against the expected detection delay,
\begin{equation}\label{eq:risk-disorder}
  R_\pi(\tau)=\PR_\pi(\tau<\theta)+c\,\E_\pi\big[(\tau-\theta)^+\big],\qquad c>0.
\end{equation}
Let $\post_t:=\PR_\pi(\theta\le t\mid\Fx^X_t)$. The false-alarm probability is
$\PR_\pi(\tau<\theta)=\E_\pi[\PR_\pi(\theta>\tau\mid\Fx^X_\tau)]=\E_\pi[1-\post_\tau]$,
and, by Fubini and the optional projection,
\begin{equation}\label{eq:delay-id}
  \E_\pi\big[(\tau-\theta)^+\big]
  =\E_\pi\!\left[\int_0^\tau\ind_{\{\theta\le s\}}\dd s\right]
  =\E_\pi\!\left[\int_0^\tau\PR_\pi(\theta\le s\mid\Fx^X_s)\dd s\right]
  =\E_\pi\!\left[\int_0^\tau\post_s\dd s\right].
\end{equation}
Hence \eqref{eq:risk-disorder} becomes the optimal stopping problem
\begin{equation}\label{eq:value-disorder}
  V(\pi)=\inf_\tau\,\E_\pi\!\left[(1-\post_\tau)+c\int_0^\tau\post_s\dd s\right],
\end{equation}
with running cost $f(p)=cp$ and terminal cost $G(p)=1-p$. The analytically
convenient Shiryaev (weighted likelihood-ratio) statistic is the posterior odds,
which admits the explicit representation
\begin{equation}\label{eq:shiryaev-stat}
  \odds_t:=\frac{\post_t}{1-\post_t}
   =\frac{\pi}{1-\pi}\,e^{\lambda t}L_t
   +\lambda\int_0^t e^{\lambda(t-s)}\,\frac{L_t}{L_s}\dd s,
\end{equation}
where $L_t/L_s$ is the post-change-to-pre-change likelihood ratio over $[s,t]$
formed from \eqref{eq:LR}. 
The filtering equation for the posterior, with the
compensator $\lambda(1-\post_t)$ of $\ind_{\{\theta\le t\}}$ induced by
\eqref{eq:prior}, is
\begin{equation}\label{eq:post-disorder-sde}
  \dd\post_t=\lambda(1-\post_t)\dd t+\snr(X_t)\,\post_t(1-\post_t)\dd\bbar_t,
\end{equation}
with $X$ as in \eqref{eq:post-test-sde}. As above, $(\odds,X)$ is the equivalent
state under the bijection $p\mapsto p/(1-p)$.

\subsection{Closed Markov state and generator}\label{sec:closed}
The reductions \eqref{eq:value-test} and \eqref{eq:value-disorder} are optimal
stopping problems driven by the posterior. Whether the posterior is by itself a
closed Markov state depends on the signal-to-noise ratio.

\begin{assumption}\label{ass:coeff}
$\mu_0,\mu_1,\sigma$ are locally Lipschitz on $I$, $\sigma>0$ on $I$, and
\eqref{eq:Xdyn} admits weakly unique nonexplosive solutions under $H_0$ and
$H_1$. Moreover, for each finite $T>0$,
\[
  \E_i\!\left[\exp\Big\{\tfrac12\int_0^T\snr^2(X_s)\dd s\Big\}\right]<\infty
  \quad(i=0,1),
\]
so that \eqref{eq:LR} is a true $\PR_0$-martingale on finite horizons.
\end{assumption}

\begin{theorem}[Closed Markov state for diffusion observations]\label{thm:closed}
Under Assumption~\ref{ass:coeff}, the pair $(\post_t,X_t)_{t\ge0}$ is a
time-homogeneous $\Filt^X$-Markov sufficient statistic for the sequential
decision problem, in both the testing and quickest-detection settings. Its
generator on $w\in C^2((0,1)\times I)$ is
\begin{equation}\label{eq:genT}
\begin{aligned}
  (\Lgen^{T}w)(p,x)
  ={}&\tfrac12\snr^2(x)\,p^2(1-p)^2 w_{pp}
   +\snr(x)\sigma(x)\,p(1-p)\,w_{px}\\
   &+\tfrac12\sigma^2(x)\,w_{xx}
   +\big[\mu_0(x)+(\mu_1-\mu_0)(x)\,p\big]\,w_x
\end{aligned}
\end{equation}
in the testing case, and $\Lgen^{D}=\Lgen^{T}+\lambda(1-p)\,\partial_p$ in the
quickest-detection case. If $\snr$ is constant, the posterior coordinate has
closed one-dimensional Markov dynamics and the problem projects onto $\post$
alone, with generator
\begin{equation}\label{eq:gen-1d}
  (\Lgen_0 w)(p)=\lambda(1-p)\,w'(p)+\tfrac12\snr^2 p^2(1-p)^2 w''(p)
\end{equation}
(omitting the $\lambda$-drift in the testing case). If $\snr$ is state dependent,
the posterior SDE is not closed in $\post$ alone. Thus $(\post,X)$ provides the
natural closed Markov realization of the filtering state. We do not attempt to
characterize exceptional projection cases in which the posterior marginal may
nevertheless be Markov.
\end{theorem}

\begin{proof}
Assumption~\ref{ass:coeff} is Novikov's criterion, so \eqref{eq:LR} is a true
martingale and Girsanov's theorem yields the odds representations of
Sections~\ref{sec:test}--\ref{sec:detect}. The innovation theorem
\citep{liptser1977statistics} gives the $\Filt^X$-Brownian motion $\bbar$ of
\eqref{eq:innov}, and the Kushner--Stratonovich equation for the two-valued hidden
variable produces \eqref{eq:post-test-sde} and, with the compensator
$\lambda(1-\post_t)$ of $\ind_{\{\theta\le t\}}$ under the prior \eqref{eq:prior},
\eqref{eq:post-disorder-sde}. Writing $X$ in the innovation gives the joint
dynamics, with quadratic covariation
\begin{equation}\label{eq:qcov}
  \dd\langle\post,X\rangle_t=\snr(X_t)\,\sigma(X_t)\,\post_t(1-\post_t)\dd t.
\end{equation}
Applying It\^o's formula to $w(\post_t,X_t)$ and collecting drift terms yields
\eqref{eq:genT} and, with the extra posterior drift, $\Lgen^{D}$; the cross term
in \eqref{eq:genT} is exactly \eqref{eq:qcov}. All coefficients are
time-independent functions of the current value $(\post_t,X_t)$, so $(\post,X)$ is
a time-homogeneous Markov process; sufficiency for the decision problem is
inherited from the posterior being a sufficient statistic for $\theta$. If $\snr$
is constant and $w$ depends on $p$ only, \eqref{eq:genT} collapses to
\eqref{eq:gen-1d} and the marginal law of $\post$ is determined by $\post$ alone.
When $\snr$ is state dependent, the coefficient of the posterior martingale term
in \eqref{eq:post-test-sde}--\eqref{eq:post-disorder-sde} depends on $X_t$, so the
posterior equation is not autonomous in $\post_t$. The augmented process supplies
a closed Markov state; possible exceptional Markovian projections are outside the
scope of the present comparison result.
\end{proof}

\begin{remark}[On the cross term and projection]\label{rem:crossterm}
The cross term $w_{px}$ in \eqref{eq:genT} reflects the common innovation noise
driving both the posterior and the observation. When $\snr$ is constant the
posterior coordinate has closed one-dimensional Markov dynamics and the stopping
problem can be projected onto $\post$ alone; the cross term persists only in the
redundant two-dimensional representation $(\post,X)$. State dependence of $\snr$
removes the projection and makes $(\post,X)$ the operative state.
\end{remark}

\begin{remark}[Degeneracy and regularity]\label{rem:degenerate}
The diffusion matrix in \eqref{eq:genT},
\[
  \begin{pmatrix}\snr^2 p^2(1-p)^2 & \snr\sigma p(1-p)\\[2pt]
  \snr\sigma p(1-p) & \sigma^2\end{pmatrix},
\]
has determinant zero, so
$(\post,X)$ diffuses along a single direction in $(p,x)$-space: both coordinates
are driven by the one innovation $\bbar$. The operator is therefore degenerate
elliptic rather than uniformly elliptic. Hypoellipticity nonetheless holds under
H\"ormander-type conditions on $(\snr,\sigma,\mu_i)$, which underlies the
regularity and continuity of the resulting two-dimensional stopping boundaries
\citep{peskir2019continuity,ernst2024gapeev}.
\end{remark}

\begin{example}[State-dependent signal-to-noise: Bessel dimension]\label{ex:bessel}
With $\sigma\equiv1$ and $\mu_i(x)=(d_i-1)/(2x)$ on $I=(0,\infty)$, the process
$X$ is a Bessel process of dimension $d_i$ under $H_i$, and
$\snr(x)=(d_1-d_0)/(2x)$ is state dependent. The likelihood ratio \eqref{eq:LR}
becomes
\[
  L_t=\exp\!\Big\{\tfrac{d_1-d_0}{2}\int_0^t X_s^{-1}\dd X_s
   -\tfrac{(d_1-d_0)(d_1+d_0-2)}{8}\int_0^t X_s^{-2}\dd s\Big\},
\]
and does not yield a closed one-dimensional posterior equation; the augmented pair
$(\post,X)$ is the closed state. The resulting two-dimensional stopping problem
admits an analytic characterization in terms of special functions and the
associated free-boundary conditions \citep{johnson2017quickest,johnson2018sequential}.
\end{example}

\section{Generic Optimal Stopping Formulation}\label{sec:os}

The reductions above are instances of a single optimal stopping problem. Let
$Y=(Y_t)_{t\ge0}$ be a time-homogeneous Markov process on a state space
$\mathcal S$ with generator $\Lgen$ (for example $Y=\post$, $(\post,X)$, or
$(\odds,X)$). Given a running cost $f\ge0$ and a terminal cost $G$, set
\begin{equation}\label{eq:generic-os}
  V(y)=\inf_\tau\,\E_y\!\left[\int_0^\tau f(Y_s)\dd s+G(Y_\tau)\right],\qquad
  y\in\mathcal S,
\end{equation}
with continuation and stopping regions
\begin{equation}\label{eq:CD}
  \mathcal C=\{y:V(y)<G(y)\},\qquad \mathcal D=\{y:V(y)=G(y)\}.
\end{equation}
For sequential testing $f\equiv1$ and $G=M$ of \eqref{eq:value-test}; for
quickest detection $f(p)=cp$ and $G(p)=1-p$ of \eqref{eq:value-disorder}.

\paragraph{Standing assumptions.}
We assume throughout the comparison results that, for the running and terminal
costs under consideration, the value function \eqref{eq:generic-os} is finite on
$\mathcal S$ and the first-entry time
$\tau^\ast=\inf\{t\ge0:Y_t\in\mathcal D\}$ is optimal in \eqref{eq:generic-os}.
These are mild and standard under, for instance, lower semicontinuity of $G$,
continuity of $f$, and a moment or transience condition ensuring finiteness; they
hold in the diffusion models considered here \citep{peskir2006optimal}.

The dynamic programming principle then gives
\begin{equation}\label{eq:dpp}
  V\le G,\qquad \Lgen V+f=0\ \text{ on }\mathcal C,\qquad
  \Lgen V+f\ge0\ \text{ on }\mathcal D,
\end{equation}
or, equivalently, the variational inequality
\begin{equation}\label{eq:VI}
  \max\big\{-\big(\Lgen V+f\big),\ V-G\big\}=0\qquad\text{on }\mathcal S,
\end{equation}
which is in turn equivalent to the complementarity system
\begin{equation}\label{eq:complementarity}
  V\le G,\qquad \Lgen V+f\ge0,\qquad (G-V)\,(\Lgen V+f)=0 .
\end{equation}
Both arguments of the maximum in \eqref{eq:VI} are nonpositive, which makes the
sign convention transparent: on $\mathcal C$ one has $\Lgen V+f=0$, and on
$\mathcal D$ one has $\Lgen V+f\ge0$, equivalently $\Lgen G+f\ge0$.\footnote{The
compact form $\min\{\Lgen V+f,\,G-V\}=0$ is equivalent to
\eqref{eq:VI}--\eqref{eq:complementarity}; we adopt the maximum/complementarity
form because both of its arguments are manifestly nonpositive and the equality
$\Lgen V+f=0$ on $\mathcal C$ is then immediate.} Interpretations are in the
classical, Sobolev, or viscosity sense according to the regularity of $V$. The
optimal rule is $\tau^\ast=\inf\{t:Y_t\in\mathcal D\}$.

\section{Delay-Penalty Comparison}\label{sec:comparison}

The central result orders sequential rules by their running cost. It is purely
comparative and does not require solving \eqref{eq:VI}.

\begin{theorem}[Delay-penalty ordering of sequential rules]\label{thm:comparison}
Let $Y$ be a Markov process on $\mathcal S$ with generator $\Lgen$, fix a common
terminal cost $G$, and for $i=1,2$ let $f_i\ge0$ be measurable running costs with
value functions $V_i$, regions $\mathcal C_i,\mathcal D_i$ as in
\eqref{eq:generic-os}--\eqref{eq:CD}, and optimal first-entry times
$\tau_i^\ast=\inf\{t:Y_t\in\mathcal D_i\}$. Assume $f_1\ge f_2$ pointwise on
$\mathcal S$. Then:
\begin{enumerate}[label=\emph{(\roman*)}]
\item \emph{(Value)} $V_1\ge V_2$ on $\mathcal S$.
\item \emph{(Regions and stopping times)}
$\mathcal C_1\subseteq\mathcal C_2$, $\mathcal D_2\subseteq\mathcal D_1$, and
$\tau_1^\ast\le\tau_2^\ast$ $\PR_y$-almost surely for every $y$. Apart from the
standing assumptions that the value functions are finite and that the displayed
first-entry times are optimal, no monotonicity, smoothness, or one-sidedness of
the stopping set is needed.
\item \emph{(Boundaries)} If, in addition, each stopping region is one-sided in
the posterior coordinate,
$\mathcal D_i=\{(p,x)\in\mathcal S:p\ge b_i(x)\}$ for boundary functions
$b_i:I\to[0,1]$---the structure established under a monotone signal-to-noise
condition by \citet{gapeev2011sequential,gapeev2013bayesian,ernst2024gapeev}---then
$b_1(x)\le b_2(x)$ for all $x\in I$.
\end{enumerate}
\end{theorem}

\begin{proof}
\emph{(i)} For each admissible $\tau$, pathwise $f_1\ge f_2\ge0$ gives
$\int_0^\tau f_1(Y_s)\dd s\ge\int_0^\tau f_2(Y_s)\dd s$, hence
$\E_y[\int_0^\tau f_1\dd s+G(Y_\tau)]\ge\E_y[\int_0^\tau f_2\dd s+G(Y_\tau)]$;
taking the infimum over $\tau$ yields $V_1(y)\ge V_2(y)$.

\emph{(ii)} Choosing $\tau\equiv0$ shows $V_i\le G$. If $y\in\mathcal C_1$, i.e.\
$V_1(y)<G(y)$, then by (i) $V_2(y)\le V_1(y)<G(y)$, so $y\in\mathcal C_2$; thus
$\mathcal C_1\subseteq\mathcal C_2$ and, complementarily,
$\mathcal D_2\subseteq\mathcal D_1$. Since $\tau_1^\ast$ and $\tau_2^\ast$ are
first-entry times of the \emph{same} process $Y$ into
$\mathcal D_1\supseteq\mathcal D_2$, any entry of $Y$ into $\mathcal D_2$ already
lies in $\mathcal D_1$; hence $\tau_1^\ast\le\tau_2^\ast$ almost surely.

\emph{(iii)} Under the one-sided representation,
$\mathcal D_2\subseteq\mathcal D_1$ reads
$\{p\ge b_2(x)\}\subseteq\{p\ge b_1(x)\}$ for each fixed $x$, which holds if and
only if $b_1(x)\le b_2(x)$.
\end{proof}

The monotonicity of the optimal stopping boundary with respect to the marginal delay penalty is intuitively depicted in Figure \ref{fig:optimal_stopping_boundary}. When the system faces a more stringent delay penalty (i.e., $c_1 > c_2$), the decision-maker becomes more conservative, which structurally shrinks the continuation region. Consequently, the optimal threshold shifts downward, yielding $b_{c_1}(x) \le b_{c_2}(x)$ for all given states $X_t$.

\begin{figure}[htbp]
    \centering
    \includegraphics[width=0.9\linewidth]{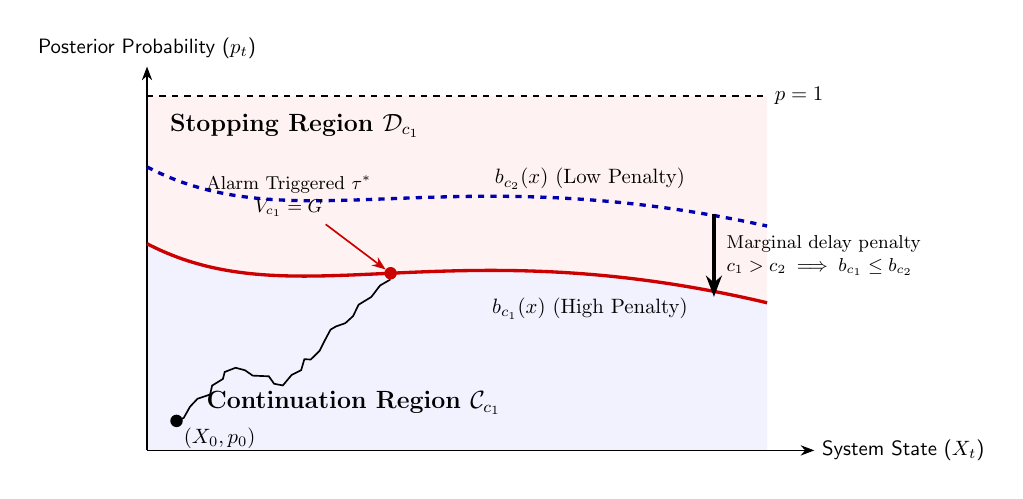}
    \caption{Illustration of the optimal stopping boundaries in the $(X_t, p_t)$ state space. The optimal policy partitions the space into a continuation region $\mathcal{C}_{c_1}$ and a stopping region $\mathcal{D}_{c_1}$, separated by the boundary $b_{c_1}(x)$. A sample trajectory of the augmented process is shown, which starts at $(X_0, p_0)$ and triggers an alarm at the optimal stopping time $\tau^*$ upon hitting the boundary. Additionally, the figure demonstrates the comparative statics: a higher marginal delay penalty ($c_1 > c_2$) strictly lowers the optimal stopping threshold, such that $b_{c_1}(x) \le b_{c_2}(x)$.}
\label{fig:optimal_stopping_boundary}
\end{figure}

\begin{remark}[Comparative-statics reading]\label{rem:comparative}
Parts (i)--(ii) hold under only the standing assumptions and deliver the
operational message: a uniformly larger marginal delay penalty makes continuation
less attractive everywhere and triggers (weakly) earlier alarms, hence shorter
expected delay at the cost of more false alarms. Part (iii) translates this into
boundary geometry, but only once the stopping region is known to be one-sided---a
property that need not hold for arbitrary state-dependent diffusions and is
exactly what the monotone signal-to-noise results secure. We do not establish the
one-sided structure here; we invoke it from
\citet{gapeev2011sequential,gapeev2013bayesian,ernst2024gapeev}.
\end{remark}

\subsection{Linear delay cost}\label{sec:linear}
The classical disorder problem has $f(p)=cp$ and $G(p)=1-p$ on the common state
$Y$. Theorem~\ref{thm:comparison} with $f_i=c_i\,p$ specializes as follows.

\begin{corollary}[Linear delay cost]\label{cor:linear}
Let $c_1\ge c_2>0$ in the Bayesian disorder problem with common state $Y$ and
terminal cost $G(p)=1-p$. Then
\[
  V_{c_1}\ge V_{c_2},\qquad
  \mathcal C_{c_1}\subseteq\mathcal C_{c_2},\qquad
  \tau_{c_1}^\ast\le\tau_{c_2}^\ast\quad\PR_y\text{-a.s.}
\]
If the stopping set has the form $\mathcal D_c=\{(p,x):p\ge b_c(x)\}$, then
$b_{c_1}(x)\le b_{c_2}(x)$ for all $x\in I$. In particular, in the constant-SNR
case the single threshold satisfies $p^\ast(c_1)\le p^\ast(c_2)$: a larger delay
cost rate implies an earlier alarm.
\end{corollary}

\subsection{Sampling cost in sequential testing}\label{sec:sampling}
The same principle applies on the testing side, where the running cost is the
sampling cost. Scaling the sampling rate to $\kappa>0$ replaces
\eqref{eq:value-test} by $V_\kappa(\pi,x)=\inf_\tau\E[\kappa\tau+M(\post_\tau)]$,
i.e.\ $f\equiv\kappa$ with common terminal cost $G=M$.

\begin{corollary}[Sampling cost]\label{cor:sampling}
Let $\kappa_1\ge\kappa_2>0$ in the sequential testing problem with common terminal
cost $G=M$. Then $V_{\kappa_1}\ge V_{\kappa_2}$,
$\mathcal C_{\kappa_1}\subseteq\mathcal C_{\kappa_2}$, and
$\tau_{\kappa_1}^\ast\le\tau_{\kappa_2}^\ast$ almost surely: a more expensive
observation stream induces an earlier terminal decision and a narrower
continuation band around the indifference point $p^\dagger$.
\end{corollary}

\subsection{Nonlinear marginal delay penalties}\label{sec:nonlinear}
For a general delay profile the risk is
$R_\pi(\tau)=\PR_\pi(\tau<\theta)+c\,\E_\pi[g((\tau-\theta)^+)]$ with $g\ge0$
nondecreasing and $g(0)=0$. The running cost is then no longer a function of
$\post_t$ alone, because the marginal delay cost at time $t$ depends on the
unobserved elapsed post-change duration $t-\theta$. The next proposition isolates
the relevant statistic.

\begin{proposition}[Marginal-cost augmentation]\label{prop:marginal}
Let $g$ be absolutely continuous, nondecreasing, with $g(0)=0$, and suppose
\[
  \E_\pi\!\left[\int_0^\tau g'(t-\theta)\,\ind_{\{\theta\le t\}}\dd t\right]<\infty
\]
for the stopping times $\tau$ under consideration. Then
\begin{equation}\label{eq:marginal-identity}
  \E_\pi\big[g((\tau-\theta)^+)\big]
   =\E_\pi\!\left[\int_0^\tau
     \E_\pi\big[g'(t-\theta)\,\ind_{\{\theta\le t\}}\mid\Fx^X_t\big]\dd t\right].
\end{equation}
Consequently the nonlinear-delay disorder problem is the optimal stopping problem
\eqref{eq:generic-os} with terminal cost $G(p)=1-p$ and running cost
\begin{equation}\label{eq:nonlinear-running}
  f_t=c\,\Psi_t,\qquad
  \Psi_t=\E_\pi\big[g'(t-\theta)\,\ind_{\{\theta\le t\}}\mid\Fx^X_t\big].
\end{equation}
\end{proposition}

\begin{proof}
By absolute continuity and $g(0)=0$,
$g((\tau-\theta)^+)=\int_0^{(\tau-\theta)^+}g'(u)\dd u
=\int_0^\tau g'(t-\theta)\ind_{\{\theta\le t\}}\dd t$ pathwise. The integrability
hypothesis legitimizes taking $\E_\pi$ and applying Fubini together with the
optional projection of the integrand onto $\Filt^X$ (valid since $\tau$ is an
$\Filt^X$-stopping time), which gives \eqref{eq:marginal-identity} and hence the
running-cost representation \eqref{eq:nonlinear-running}.
\end{proof}

Whether $f_t=c\Psi_t$ yields a finite-dimensional Markov stopping problem depends
on $g$ and is not automatic. Three cases are representative.

\emph{(a) Linear delay,} $g(u)=u$. Then $g'\equiv1$ and
$\Psi_t=\PR_\pi(\theta\le t\mid\Fx^X_t)=\post_t$, recovering $f=c\post_t$ on the
state of Theorem~\ref{thm:closed}.

\emph{(b) Exponential delay,} $g(u)=(e^{\beta u}-1)/\beta$ with $\beta>0$. Then
$g'(u)=e^{\beta u}$ and the process $\Psi_t$ can be represented through a weighted
likelihood-ratio statistic,
$\Psi_t=e^{\beta t}\,\E_\pi[e^{-\beta\theta}\ind_{\{\theta\le t\}}\mid\Fx^X_t]$.
In state-dependent diffusion models this statistic must still be combined with the
observation state $X_t$ to obtain a closed Markov state; the representation is not
a universal dimension reduction, and the resulting alarm boundary is generally
observation-dependent \citep{gapeev2013bayesian}.

\emph{(c) General $g$.} The statistic $\Psi_t$ need not admit a finite-dimensional
filter, and the state must be augmented with accumulated-penalty information for
the stopping problem to be Markovian.

In every case in which a common Markov state $Y$ carries the statistics for two
penalties $g_1,g_2$, the pointwise ordering of marginal penalties
$g_1'\ge g_2'$ transfers, via \eqref{eq:nonlinear-running} and the monotonicity of
conditional expectation, to $f_1\ge f_2$. Theorem~\ref{thm:comparison} then
applies and yields the corresponding ordering of value functions, continuation
regions, stopping times, and---where the one-sided structure holds---alarm
boundaries. The convex (e.g.\ exponential) case has $g'$ increasing, so a larger
$\beta$ produces a pointwise-larger marginal penalty and an earlier alarm than the
linear benchmark.

\section{Free-Boundary Interpretation and Verification}\label{sec:fb}

On the Markov state $Y=(\post,X)$ (or its one-dimensional reduction),
\eqref{eq:VI} is a free-boundary problem: $\Lgen V+f=0$ on $\mathcal C$, $V=G$ on
$\mathcal D$, with matching conditions across the free boundary
$\partial\mathcal C$. The \emph{continuous-fit} condition
$V|_{\partial\mathcal C}=G|_{\partial\mathcal C}$ always holds; the
\emph{smooth-fit} condition $\nabla V|_{\partial\mathcal C}=\nabla
G|_{\partial\mathcal C}$ holds when the boundary point is probabilistically
regular for the interior of $\mathcal D$ and the diffusion is nondegenerate there.
For regular one-dimensional diffusions smooth fit is standard
\citep{peskir2006optimal}. In the present degenerate two-dimensional setting it
may fail where the diffusion coefficient $\snr(x)\,p(1-p)$ vanishes (at $p\in\{0,1\}$
or where $\snr(x)=0$) or where the boundary is otherwise irregular; there only
continuous fit is available, and the boundary's continuity is itself a delicate
question \citep{peskir2019continuity}.

When $\snr$ is constant the free-boundary problem reduces to the ordinary
differential equation $\Lgen_0 V+f=0$ on the continuation interval, with
$\Lgen_0$ of \eqref{eq:gen-1d}, and the boundary is a single threshold determined
by smooth fit; this is the setting of Section~\ref{sec:example}. When $\snr$ is
state dependent, the operator is degenerate elliptic in $(p,x)$,
$\partial\mathcal C$ is a curve, and---consistent with the two-boundary structure
found by \citet{gapeev2011sequential}---the alarm is the first exit of the posterior
from a region bounded by \emph{observation-dependent} (stochastic) boundaries;
explicit solutions exist only in special cases, and the boundary is otherwise
characterized by systems of nonlinear integral equations arising from the
change-of-variable formula with local time on curves, or computed numerically.

A candidate solution of \eqref{eq:VI} is confirmed optimal by martingale
verification.

\begin{proposition}[Verification]\label{prop:verif}
Let $\widehat V$ be continuous on $\mathcal S$, of polynomial growth, $C^1$ across
$\partial\widehat{\mathcal C}$, and $C^2$ on the interiors of
$\widehat{\mathcal C}=\{\widehat V<G\}$ and $\widehat{\mathcal D}=\{\widehat V=G\}$,
and suppose
\[
  \Lgen\widehat V+f\ge0\ \text{on }\mathcal S,\qquad
  \widehat V\le G\ \text{on }\mathcal S,\qquad
  \Lgen\widehat V+f=0\ \text{on }\widehat{\mathcal C}.
\]
Suppose moreover that for every admissible $\tau$ the local martingale
$M_{\,\cdot\,}$ in \eqref{eq:ito-verif} below, stopped along a localizing sequence
$\tau_n\uparrow\infty$, is uniformly integrable in the limit (e.g.\ a
square-integrability or sublinear-growth condition on
$\nabla\widehat V\cdot\sigma(Y)$). Then $\widehat V=V$, and
$\tau^\ast=\inf\{t:Y_t\in\widehat{\mathcal D}\}$ is optimal whenever
$\E_y[\tau^\ast]<\infty$.
\end{proposition}

\begin{proof}
For an admissible $\tau$ and a localizing sequence $\tau_n\uparrow\infty$, the
It\^o--Tanaka formula applied to $\widehat V(Y_{\cdot})$ gives
\begin{equation}\label{eq:ito-verif}
  \widehat V(Y_{\tau\wedge\tau_n})+\int_0^{\tau\wedge\tau_n} f(Y_s)\dd s
   =\widehat V(y)+\int_0^{\tau\wedge\tau_n}(\Lgen\widehat V+f)(Y_s)\dd s
   +M_{\tau\wedge\tau_n},
\end{equation}
where $M$ is a local martingale. The $C^1$ (smooth-fit) hypothesis across
$\partial\widehat{\mathcal C}$ ensures that the local-time term on the free
boundary, which would otherwise appear because $\Lgen\widehat V$ has a jump in its
second derivatives there, vanishes; where only continuous fit holds, the
local-time term is nonnegative and is retained in the inequality below without
affecting its direction. Since $\Lgen\widehat V+f\ge0$, taking expectations and
letting $n\to\infty$ (using the growth and uniform-integrability hypotheses so
that $\E_y[M_{\tau\wedge\tau_n}]\to0$) gives
$\widehat V(y)\le\E_y[\int_0^\tau f\dd s+\widehat V(Y_\tau)]
\le\E_y[\int_0^\tau f\dd s+G(Y_\tau)]$, hence $\widehat V\le V$. For
$\tau=\tau^\ast$ the integrand $\Lgen\widehat V+f$ vanishes on
$\widehat{\mathcal C}$ and $\widehat V(Y_{\tau^\ast})=G(Y_{\tau^\ast})$, so the
inequalities are equalities and $\widehat V(y)=\E_y[\int_0^{\tau^\ast}f\dd
s+G(Y_{\tau^\ast})]\ge V(y)$. Thus $\widehat V=V$ and $\tau^\ast$ is optimal.
\end{proof}

\section{Worked Example: Threshold Monotonicity in the Shiryaev Diffusion Model}
\label{sec:example}

We illustrate Corollary~\ref{cor:linear} in the constant-SNR quickest-detection
model, where the posterior is a closed one-dimensional diffusion
\begin{equation}\label{eq:shiryaev-diffusion}
  \dd\post_t=\lambda(1-\post_t)\dd t+\rho\,\post_t(1-\post_t)\dd\bbar_t,
  \qquad \rho:=\snr=\text{const}.
\end{equation}
With running cost $f(p)=cp$ and terminal cost $G(p)=1-p$, the value function
solves the variational inequality \eqref{eq:VI} with the one-dimensional generator
\eqref{eq:gen-1d}. On the continuation region the equation $\Lgen_0 V+cp=0$ reads
\begin{equation}\label{eq:example-ode}
  \tfrac12\rho^2 p^2(1-p)^2\,V''(p)+\lambda(1-p)\,V'(p)+c\,p=0,
\end{equation}
and on the stopping region $V=G$ with $\Lgen_0 G+cp=cp-\lambda(1-p)\ge0$ required
by \eqref{eq:complementarity}, i.e.\ $p\ge\lambda/(c+\lambda)$.

\paragraph{Numerical method.}
We solve the obstacle problem \eqref{eq:complementarity} numerically rather than
relying on a closed entrance-boundary shooting condition. The interval $[0,1]$ was
discretized with a uniform grid of size $\Delta p$. The degenerate diffusion
coefficient $\tfrac12\rho^2 p^2(1-p)^2$ was evaluated at grid points and
discretized by central differences; the drift term $\lambda(1-p)\ge0$ was upwinded
(forward difference), which renders the discrete generator a monotone $M$-matrix.
The degenerate left end $p=0$, where the diffusion coefficient vanishes and the
drift is $\lambda>0$, supplies its own discrete relation $V_0=V_1$ through the
upwinded operator, so no entrance-boundary derivative condition is imposed by
hand; at $p=1$ we set $V=G=0$. The obstacle problem was then solved by projected
(policy) iteration---each sweep performs a Gauss--Seidel/successive-overrelaxation
update of $\Lgen_h V+f=0$ followed by the projection $V\leftarrow\min\{V,G\}$
onto the obstacle---iterated until the active set stabilized and the update fell
below $10^{-11}$. The reported thresholds were stable under halving of $\Delta p$
(grids of $500,1000,2000,4000$ points agree to the displayed digits) and were
cross-checked against an entrance-boundary shooting solution of
\eqref{eq:example-ode}; the two methods agree.

\paragraph{Results.}
Table~\ref{tab:thresholds} reports the optimal posterior threshold $p^\ast(c)$ for
$\lambda=0.05$ and $\rho=1.0$, together with the implied likelihood-ratio
threshold $\odds^\ast=p^\ast/(1-p^\ast)$. The threshold decreases monotonically as
the delay cost rate $c$ increases. One checks directly that the computed
thresholds satisfy $p^\ast\ge\lambda/(c+\lambda)$, so the candidate solves the
variational inequality and is optimal by Proposition~\ref{prop:verif}.
Figure~\ref{fig:thresholds} shows the value functions peeling away from the
obstacle---continuation regions shrinking as $c$ grows---and the monotone curve
$c\mapsto p^\ast(c)$. We emphasize that the monotonicity observed in the computed
thresholds is not used as evidence for Corollary~\ref{cor:linear}; it only
illustrates the theorem, which is proved independently in
Section~\ref{sec:comparison}.

\begin{table}[t]
\centering
\begin{tabular}{cccc}
\toprule
$\lambda$ & $\rho$ & $c$ & threshold $p^\ast(c)$ \ \ ($\odds^\ast=p^\ast/(1-p^\ast)$)\\
\midrule
0.05 & 1.0 & 0.5 & $0.1735$ \ \ $(0.2099)$\\
0.05 & 1.0 & 1.0 & $0.0705$ \ \ $(0.0759)$\\
0.05 & 1.0 & 2.0 & $0.0303$ \ \ $(0.0313)$\\
0.05 & 1.0 & 5.0 & $0.0109$ \ \ $(0.0110)$\\
\bottomrule
\end{tabular}
\caption{Monotone decrease of the Shiryaev alarm threshold as the delay cost rate
$c$ increases, illustrating Corollary~\ref{cor:linear}
($p^\ast(c_1)\le p^\ast(c_2)$ when $c_1\ge c_2$). Values from the
finite-difference solution of the variational inequality, stable under halving of
the grid spacing.}
\label{tab:thresholds}
\end{table}

\begin{figure}[t]
\centering
\includegraphics[width=\textwidth]{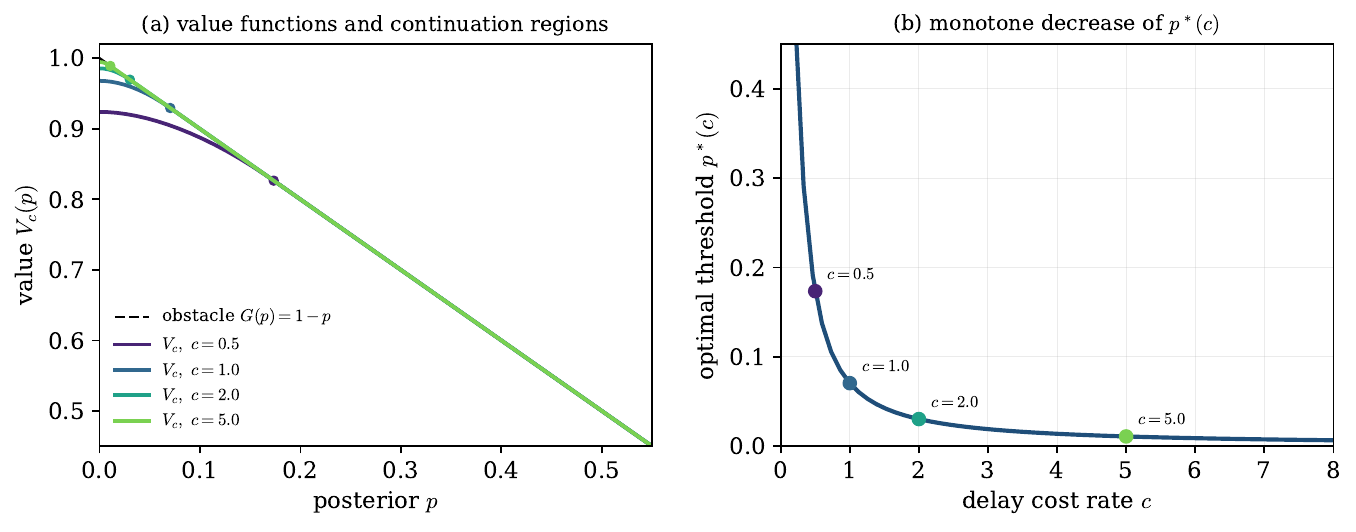}
\caption{Constant-SNR Shiryaev model ($\lambda=0.05$, $\rho=1.0$).
(a) Value functions $V_c$ below the obstacle $G(p)=1-p$; the smooth-fit points
$p^\ast(c)$ (markers) move left and the continuation region shrinks as $c$
increases. (b) The optimal threshold $p^\ast(c)$ is monotone decreasing in the
delay cost rate $c$, as predicted by Corollary~\ref{cor:linear}.}
\label{fig:thresholds}
\end{figure}

\paragraph{Interpretation.}
The comparative-statics content is direct. A higher per-unit delay cost makes the
observer less willing to wait, so the alarm is raised at a lower posterior
probability of disorder: the threshold $p^\ast(c)$, and with it the
likelihood-ratio threshold $\odds^\ast(c)$, decreases in $c$. By
Theorem~\ref{thm:comparison}(ii) the associated alarm times are ordered pathwise,
$\tau^\ast_{c_1}\le\tau^\ast_{c_2}$ for $c_1\ge c_2$, so a costlier delay yields
uniformly earlier alarms; the price is a higher false-alarm probability, since
stopping at a lower posterior is more often premature. The same qualitative
picture persists for state-dependent $\snr$, where $p^\ast$ is replaced by an
observation-dependent boundary $b_c(x)$ ordered as in
Theorem~\ref{thm:comparison}(iii).

\section{Relation to CUSUM and Shiryaev--Roberts Procedures}\label{sec:cusum}

The present comparison result is Bayesian and optimal-stopping based. It is
therefore closest to the Shiryaev procedure and its limiting Shiryaev--Roberts
forms \citep{shiryaev1963optimum,pollak1985optimal,pollak2009optimality}, and it is not a
minimax optimality statement for CUSUM \citep{lorden1971procedures,moustakides1986optimal}.
Nevertheless, the same state-dependence issue appears in the minimax formulations:
when the likelihood increments
$\dd\log L_t=-\tfrac12\snr^2(X_t)\dd t+\snr(X_t)\dd\bbar_t$ depend on the current
diffusion state, the CUSUM and Shiryaev--Roberts statistics are not closed
one-dimensional Markov processes unless the observation state $X_t$ is included,
and exactly characterized rules then involve the joint process $(\,\cdot\,,X)$. A
comparison principle for the minimax thresholds analogous to
Theorem~\ref{thm:comparison} would require monotonicity of the worst-case
detection delay in the penalty parameters and is left for future work.

\section{Conclusion}\label{sec:concl}

We have organized sequential testing and Bayesian quickest detection for
state-dependent diffusion observations around two facts. The first is a closed
Markov-state reduction identifying $(\post,X)$ as the sufficient statistic
whenever the signal-to-noise ratio is state dependent, with an explicit degenerate
generator. The second, and the methodological core of the paper, is a
delay-penalty comparison theorem: uniformly larger marginal delay costs raise the
value, shrink the continuation region, order the stopping times pathwise, and---under
a one-sided boundary representation---lower the alarm boundary. The same principle
gives a sampling-cost comparison for sequential testing. A constant-SNR Shiryaev
example, solved through the variational inequality, exhibits the predicted
threshold monotonicity. The comparison is deliberately structural rather than
constructive: it presumes a stopping formulation and, for the boundary statement,
the one-sided structure secured by monotone signal-to-noise conditions
\citep{ernst2024gapeev}, and it does not provide new explicit boundary solutions.
Natural extensions include multi-source and multi-hypothesis detection, where the
posterior lives on a simplex and the stopping regions are separated by
hypersurfaces \citep{dayanik2008multisource}; nonlinear penalties whose marginal
statistic \eqref{eq:nonlinear-running} requires genuine state augmentation; and
minimax analogues of the comparison principle. In each case the comparison
continues to apply whenever a common Markov state and terminal cost are available.

\bibliography{reference}

\end{document}